\documentclass[12pt]{amsart}
\usepackage{amsmath, amssymb, stmaryrd, setspace, nicefrac, multicol, amsthm, geometry,endnotes}
\usepackage{fullpage}
\usepackage[utf8]{inputenc}
\usepackage{mathrsfs}
\usepackage{stmaryrd}

\newtheorem{thm}{Theorem}

\newtheorem*{ithm}{Intersection Theorem}
\newtheorem{lem}[thm]{Lemma}
\newtheorem{cor}[thm]{Corollary}

\theoremstyle{definition}

\newtheorem{quest}[thm]{Question}

\newcommand{\cs}{2^\omega}
\newcommand{\str}{2^{<\omega}}
\newcommand{\dom}{\mathrm{dom}}

\newcommand{\uh}{{\upharpoonright}}

\newcommand{\MLR}{\mathsf{MLR}}

  \newcommand{\K}{\mathcal{K}}

 \newcommand{\llb}{\llbracket}
\newcommand{\rrb}{\rrbracket}
\renewcommand{\dim}{\mathit{dim}}

\newcommand{\Tree}{\mathsf{Tree}}

\renewcommand{\K}{\mathcal{K}}
\renewcommand{\dim}{\mathit{dim}}

\usepackage{xcolor}	
	\usepackage{soul}

\title{The Intersection of Algorithmically Random Closed Sets and Effective Dimension}
\author{Adam Case and Christopher P. Porter}
\begin{document}

\begin{abstract}
In this article, we study several aspects of the intersections of algorithmically random closed sets.  First, we answer a question of Cenzer and Weber, showing that the operation of intersecting relatively random closed sets (with respect to certain underlying measures induced by Bernoulli measures on the space of codes of closed sets), which preserves randomness, can be inverted:  a random closed set of the appropriate type can be obtained as the intersection of two relatively random closed sets.   We then extend the Cenzer/Weber analysis to the intersection of multiple random closed sets, identifying the Bernoulli measures with respect to which the intersection of relatively random closed sets can be non-empty.  We lastly apply our analysis to provide a characterization of the effective Hausdorff dimension of sequences in terms of the degree of intersectability of random closed sets that contain them.
\end{abstract}

\maketitle 

\section{Introduction}

The goal of this article is twofold.  First, we extend work on Cenzer and Weber \cite{CenWeb13} concerning the intersection of algorithmically random closed subsets of $\cs$ to provide an analysis of multiple intersections of algorithmic random closed sets.  Second, we apply our results on multiple intersections to reveal a hitherto undetected relationship between what we call a degree of intersectability of a family of random closed sets and the effective Hausdorff dimension of members of these random closed sets.  In particular, we prove that for a family of random closed sets with respect to an underlying probability measure of a certain form (known as a symmetric Bernoulli measure, defined below), there is a fixed degree of intersectability of the random closed sets in this family, and this degree is inversely related to a lower bound on the effective Hausdorff dimension of members of these random closed sets.  Moreover, given any sequence $X\in\cs$ of positive effective Hausdorff dimension, any random closed set (with respect to the relevant underlying probability measure) that contains $X$ must have a degree of intersectability that is inversely proportional to the effective dimension of $X$.

The study of algorithmically random closed sets was initiated by Barmpalias, Brodhead, Cenzer, Dashti, and Weber in \cite{BarBroCen07}.   In this study, each closed subset of $\cs$ is coded as a member of $3^\omega$, where each value in the sequence is determined by the type of branching that occurs at each node of the underlying tree corresponding to the closed set in question (we discuss the coding mechanism in Section \ref{subsec-rcs} below). The standard machinery of algorithmically random sequences then directly transfers over to the setting of closed sets.   Subsequent work on this topic was carried out in, for instance, \cite{BroCenTos11},\cite{DiaKjo12}, and \cite{CenWeb13}, and more recently, \cite{Axo15}, \cite{CulPor16}, and \cite{Axo18}.

In follow-up to their initial work, Cenzer and Weber \cite{CenWeb13} studied the unions and intersections of random closed sets with respect to a general family of measures on the space of closed subsets of $\cs$ (we write this space as $\mathcal{K}(\cs)$).  Such measures are induced by Bernoulli measures on $3^\omega$:  for $p,q$ such that $0\leq p+q\leq 1$, we define the measure $\mu_{\langle p,q\rangle}$ to satisfy, for every $\sigma\in 3^{<\omega}$,
\begin{itemize}
\item $\mu_{\langle p,q\rangle}(\sigma 0)=p\cdot \mu_{\langle p,q\rangle}(\sigma)$,
\item $\mu_{\langle p,q\rangle}(\sigma 1)=q\cdot \mu_{\langle p,q\rangle}(\sigma)$, and
\item $\mu_{\langle p,q\rangle}(\sigma 2)=(1-p-q)\cdot \mu_{\langle p,q\rangle}(\sigma)$.
\end{itemize}
We write the measure on $\K(\cs)$ induced by $\mu_{\langle p,q\rangle}$ as $\mu_{\langle p,q\rangle}^*$ (where we follow the convention first laid out in \cite{BarBroCen07} that if $\mu$ is a measure on codes of closed subsets of $\cs$, then $\mu^*$ is the induced measure on $\K(\cs)$).

Of the results obtained by Cenzer and Weber in \cite{CenWeb13}, the most relevant to the present study is what we will refer to as the \emph{Intersection Theorem}, which provides a full characterization, in terms of the parameters of Bernoulli measures on $3^\omega$, of when the associated notions of random closed sets can yield non-empty intersections:

 \begin{ithm}[Cenzer/Weber \cite{CenWeb13}]\label{thm-cw1}
Suppose that $p, q, r, s\geq 0$, $0\leq p + q\leq 1$ and $0\leq r + s\leq 1$. Suppose that  $P\in \K(\cs)$ is $\mu^*_{\langle p,q\rangle}$-random relative to $Q\in \K(\cs)$ and that $Q$ is  $\mu^*_{\langle r,s\rangle}$-random relative to $P$.
\begin{enumerate}
\item If $p+q+r+s\geq1+pr+qs$, then $P\cap Q=\emptyset$.
\item If $p+q+r+s<1+pr+qs$, then $P\cap Q=\emptyset$ with probability $\dfrac{ps+qr}{(1-p-q)(1-r-s)}$.
\item If $p+q+r+s<1+pr+qs$ and $P\cap Q\neq\emptyset$, then $P\cap Q$ is Martin-L\"of random with respect to the measure $\mu^*_{\langle p+r-pr,q+s-qs \rangle}$.
\end{enumerate}
\end{ithm}

As a corollary of the Intersection Theorem, by setting $p=q=r=s$ (obtaining what we refer to as a \emph{symmetric Bernoulli measure} on $3^\omega$, which we write as $\mu_p$, with $\mu_p^*$ standing for the corresponding measure on $\K(\cs)$), Cenzer and Weber obtain:
\begin{cor}[Cenzer/Weber \cite{CenWeb13}]\label{cor-cw}
For $p\in(0,\frac{1}{2})$, let $P,Q\in\K(\cs)$ be relatively $\mu^*_p$-random.
\begin{enumerate}
\item If $p\geq 1-\frac{\sqrt{2}}{2}$, then $P\cap Q=\emptyset$.
\item If $p< 1-\frac{\sqrt{2}}{2}$, then $P\cap Q=\emptyset$ with probability $\frac{2p^2}{(1-2p)^2}$.
\item If $p< 1-\frac{\sqrt{2}}{2}$ and $P\cap Q\neq\emptyset$, then $P\cap Q$ is Martin-L\"of random with respect to the measure $\mu^*_{2p-p^2}$.
\end{enumerate}
\end{cor}

In our analysis, we extend the work of Cenzer and Weber on the intersection of random closed sets in two respects.  First, Cenzer and Weber leave open whether a converse of the Intersection Theorem holds:

\begin{quest}\label{q2}
Suppose that $p, q, r, s\geq 0$ satisfy $0\leq p + q\leq 1$, $0\leq r + s\leq 1$ and $p+q+r+s < 1+pr+qs$ and $R$ is Martin-L\"of random with respect to the measure $\mu^*_{\langle p+r-pr,q+s-qs\rangle}$. Do there exist $P,Q\in \K(\cs)$ such that $R = P \cap Q
$, $P$ is $\mu^*_{\langle p,q\rangle}$-Martin-L\"of random and $Q$ is $\mu^*_{\langle r,s\rangle}$-Martin-L\"of random?
\end{quest}

Here we answer this question in the affirmative.   Our result makes use of an alternative characterization of $\mu^*_{\langle p,q\rangle}$-random closed sets in terms of Galton-Watson trees, generalizing a result of Kjos-Hanssen and Diamondstone \cite{DiaKjo12}.  We also use an approach similar to one due to Bienvenu, Hoyrup, and Shen \cite{BieHoyShe17}, who reprove the above result of Kjos-Hanssen and Diamondstone using the machinery of layerwise computability.

The second respect in which we extend Cenzer and Weber's work on the intersection of random closed sets pertains to multiple intersections of random closed sets.  Here we postpone the full statement of our result until more machinery has been developed, but the general idea is as follows:  From Corollary \ref{cor-cw}, we can conclude:
\begin{enumerate}
\item If $p<1-\frac{1}{\sqrt{2}}$, then relatively $\mu_p^*$-random closed sets may have a non-empty intersection.
\item If $p\geq 1-\frac{1}{\sqrt{2}}$, then relatively $\mu_p^*$-random closed sets must have an empty intersection.
\end{enumerate}
We extend this result by showing, for $n\geq 2$, the following:
\begin{enumerate}
\item If $p<1-\frac{1}{\sqrt[n]{2}}$, then $n$ mutually $\mu_p^*$-random closed sets may have a non-empty intersection.
\item If $p\geq 1-\frac{1}{\sqrt[n]{2}}$, then $n$ mutually $\mu_p^*$-random closed sets must have an empty intersection.
\end{enumerate}
Here, a sequence of closed sets is \emph{mutually $\mu_p^*$-random} if the code for each closed set in the sequence is $\mu_p$-random relative to the join of the codes of the remaining closed sets in the sequence.   We also answer the analogue of Question \ref{q2} for the intersection of $n$ random closed sets in the more general case that $n\geq 2$.  

Lastly, we apply our result on multiple intersections to obtain a new characterization of the effective dimension of members of random closed sets.  To do so, we draw on work of Diamondstone and Kjos-Hanssen on the effective Hausdorff dimension of members of random closed sets.  In particular, from results of Diamondstone and Kjos-Hanssen we can immediately conclude:
\begin{itemize}
\item[(1)] the dimension spectrum of members of $\mu^*_p$-random closed sets is $[-\log(1-p),1]$; and
\item[(2)] in the case that $p=1-\frac{1}{\sqrt[n]{2}}$, this dimension spectrum evaluates to $[\frac{1}{n},1]$.
\end{itemize}
Combining the observations with our results on multiple intersections, we can show:
\begin{itemize}
\item[(3)] the lower bound on the dimension spectrum of a family of random closed sets is inversely proportional to an upper bound on the number of mutually random closed sets that can have a non-empty intersection; and
\item[(4)] the effective dimension of a sequence is inversely proportional to the degree of intersectability of any random closed set containing it, where this degree of intersectability measures the number of mutually random closed sets of a given type that can have a non-empty intersection.
\end{itemize}
(More precise statements of these results can be found in Section \ref{sec-dimension}.)

The outline of the remainder of this paper is as follows.  First, we cover the necessary background in Section \ref{sec-bg}.  Next, Section \ref{sec-intersection} contains a proof of the converse of the Intersection Theorem (as well as a new proof of the Intersection Theorem that enables us to prove the converse).  In Section \ref{sec-multiple}, we turn to multiple intersections of random closed sets, establishing analogues of the Intersection Theorem and its converse for the intersection of any finite number of sufficiently random closed sets.  Lastly, we conclude in Section \ref{sec-dimension} with a discussion of the relationship between effective dimension and multiple intersections of random closed sets.

\section{Background}\label{sec-bg}

\subsection{Some topological and measure-theoretic basics}

As we will work with binary, ternary, and quaternary sequences in this study, we will introduce the spaces of such sequences in more generality.  For $n\in\omega$, we will write the set of all finite strings
over the alphabet $\{0,1,\ldots n-1\}$ as $n^{<\omega}$.  We use $\epsilon$ to stand for the empty string. Similarly, the space of all infinite sequences
over the alphabet $\{0,1,\ldots n-1\}$ is written $n^\omega$. For $x,y\in n^\omega$, $x\oplus y$ is the sequence $z\in n^\omega$ satisfying $z(2k)=x(k)$ and $z(2k+1)=y(k)$ for every $k \in\omega$. We similarly define $\sigma \oplus \tau$ for $\sigma,\tau \in n^{< \omega}$ where $|\sigma| = |\tau|$.

We will work with the topology on $n^{\omega}$ generated by the
clopen sets 
\[
\llb\sigma\rrb= \{x \in n^{\omega}: x\succ \sigma\},
\] where
$\sigma \in n^{<\omega}$ and $x \succ \sigma$ means that $\sigma$ is
an initial segment of $x$. For $x\in n^\omega$ and $k\in\omega$, $x\restriction k$ 
stands for the initial segment of $x$ of length $k$. For $\sigma,\tau\in n^{<\omega}$,  the concatenation
of $\sigma$ and $\tau$ will be written as $\sigma^{\frown}\tau$ or, in some cases, as $\sigma\tau$.

We say that $T\subseteq n^{<\omega}$ is a \emph{tree} if, whenever $\tau \in T$ and $\sigma \preceq \tau$, we have  $\sigma
\in T$. A \emph{path}
through a tree $T\subseteq n^{<\omega}$ is a sequence $x \in n^{\omega}$
satisfying $x \restriction k \in T$ for every $k$. The set of all
paths through a tree $T$ is denoted by $[T]$. Recall that a set $C\subseteq n^{\omega}$ is closed if and only if $C=[T]$ for
  some tree $T\subseteq n^{<\omega}$. Moreover, $C$ is non-empty if and
  only if $T$ is infinite.

A \emph{measure} $\mu$ on $n^{\omega}$ is a function that assigns to
each Borel subset of $n^{\omega}$ a value in 
$[0,1]$ and satisfies the condition $\mu(\bigcup_{i \in \omega} B_{i})= \sum_{i \in
  \omega} \mu(B_{i})$ for any pairwise disjoint sequence $(B_i)_{i\in\omega}$ of Borel sets.
By Carath\'{e}odory's extension theorem, the conditions
\begin{itemize}
  \item[(i)] $\mu(n^\omega)= 1$ and
  \item[(ii)] $\mu(\llb\sigma\rrb) = \mu(\llb\sigma0\rrb) +
  \mu(\llb\sigma1\rrb) + \ldots + \mu(\llb\sigma^{\frown}(n-1)\rrb) $ for
  all $\sigma \in n^{<\omega}$
\end{itemize}
uniquely determine a measure on $n^{\omega}$.  We often identify a measure with
a function $\mu\colon n^{<\omega} \to [0,1]$
satisfying the conditions (i) and (ii).   For each $\sigma\in n^{<\omega}$, we often write $\mu(\sigma)$
instead of $\mu(\llb\sigma\rrb)$.  The Lebesgue measure $\lambda$
on $n^\omega$ is defined by $\lambda(\sigma) = n^{-|\sigma|}$ for each string
$\sigma\in n^{<\omega}$.

Given a measure $\mu$ on $n^\omega$ and $\sigma,\tau\in n^{<\omega}$, the conditional measure
$\mu(\sigma\tau\mid\sigma)$ is defined by setting
\[
\mu(\sigma\tau\mid\sigma)=\dfrac{\mu(\sigma\tau)}{\mu(\sigma)}.
\]

\subsection{Some computability theory}
\label{sec:some-comp-theory}

We assume the reader is familiar with the basic concepts of computability
theory as found, for instance, in the early chapters of \cite{Soa16}.

A $\Sigma^0_1$ \emph{class} $S\subseteq n^\omega$ is an effectively 
open set, i.e., an effective union of basic clopen subsets of $n^\omega$.
$P\subseteq n^\omega$ is a $\Pi^0_1$ \emph{class} if $\cs\setminus P$ is a 
$\Sigma^0_1$ class.  For $n,m\in\omega$, a Turing functional $\Phi\colon \subseteq n^{\omega} \to m^{\omega}$ is defined
in terms of a computably enumerable set of pairs $S_\Phi\subseteq n^{<\omega}\times m^{<\omega}$
with the condition that if $(\sigma,\tau),(\sigma',\tau')\in S_\Phi$ and $\sigma\preceq\sigma'$, then
either $\tau\preceq\tau'$ or $\tau'\preceq\tau$.  For each $\sigma\in n^{<\omega}$, we define $\Phi^\sigma$ to be the maximal string  in $\{\tau\in m^{<\omega}: (\exists \sigma'\preceq\sigma)((\sigma',\tau)\in\Phi)\}$ in the order given by $\preceq$. To obtain a map defined on $n^\omega$ from the c.e.\ set of pairs $S_\Phi$, for each $x\in n^\omega$, we let $\Phi^x$ be the maximal $y\in m^{<\omega}\cup m^\omega$ in the order given by $\preceq$  such that $\Phi^{x \uh k}$ is a prefix of~$y$ for all~$k\in\omega$. We will thus set
$\dom(\Phi)=\{x\in n^\omega:\Phi^x\in m^\omega\}$.  When $\Phi^x\in m^\omega$, we will sometimes write $\Phi^x$ as $\Phi(x)$ to emphasize the functional $\Phi$ as a map from~$n^\omega$ to~$m^\omega$.  It is straightforward to relativize the notion 
of a Turing functional $\Phi\colon \subseteq n^{\omega} \to m^{\omega}$
to any oracle $z\in\cs$ to obtain a $z$-\emph{computable} functional.

A measure $\mu$ on $n^\omega$ is computable if $\mu(\sigma)$ 
is a computable real number, uniformly in $\sigma\in n^{<\omega}$. 
Clearly, the Lebesgue measure $\lambda$ on $n^\omega$ is computable. If $\mu$ is a computable measure on $n^\omega$ and 
$\Phi\colon \subseteq n^{\omega} \to m^{\omega}$ is a Turing functional
defined on a set of $\mu$-measure one, then the \emph{pushforward measure}
$\mu_\Phi$ defined by
\[
\mu_\Phi(\sigma)=\mu(\Phi^{-1}(\llb \sigma \rrb)),
\]
for each $\sigma\in m^{<\omega}$, is a computable measure.

\subsection{Algorithmically random sequences}  \label{subsec-mlr}

In this section, we lay out the main definitions of algorithmic randomness with which we will be working.
For more details, see \cite{Nie09}, \cite{DowHir10}, or \cite{SheUspVer17}.  See also \cite{FraPor20} for an up-to-date survey on algorithmic randomness.

  Let $\mu$ be a computable measure on $n^{\omega}$ and let $z\in  m^{\omega}$. 
Recall that a \emph{$\mu$-Martin-L\"of test relative to $z$} (or simply a \emph{$\mu$-test relative to $z$}) is a uniformly $\Sigma^{0,z}_{1}$ sequence 
   $(U_{i})_{i\in\omega}$ of subsets  of $n^{\omega}$  with $\mu (U_{n}) \leq 2^{-n}$.  Then $x\in n^\omega$ passes such a test $(U_{i})_{i\in\omega}$ if $x\notin \bigcap_{n} U_{n}$, and
$x\in n^\omega$ is $\mu$-Martin-L\"of random relative to $z$ if $x$ passes every $\mu$-Martin-L\"of test relative to $z$. The collection 
  of $\mu$-random sequences relative to $z$ will be denoted by $\MLR^{z}_{\mu}$.  When $z$ is computable we 
  simply write $\MLR_{\mu}$ and will refer to $x$ as $\mu$-random.

It is not difficult to see that if $\mu$ is a computable measure on $n^{\omega}$ and $\Phi\colon
  \subseteq n^{\omega} \to m^{\omega}$ is a Turing functional that satisfies $\mu(\dom
  (\Phi))=1$, then $\MLR_{\mu}\subseteq \dom(\Phi)$.  One of the central tools that we will use in this study is the following.  

\begin{thm}\label{prnr}
  Let $\Phi\colon \subseteq n^\omega \to m^\omega$ be a Turing functional that satisfies $\mu
  (\dom(\Phi))=1$.
  \begin{enumerate}
    \item (Preservation of Randomness \cite{ZvoLev70}) If $x \in \MLR_{\mu}$ then $\Phi(x) \in \MLR_{\mu_\Phi}$.
    \item (No Randomness from Non-Randomness \cite{She86}) If $y \in \MLR_{\mu_\Phi}$, then there is $x\in
    \MLR_{\mu}$ such that $\Phi(x)=y$.
  \end{enumerate}
  \label{thm:PoR-and-NReN}
\end{thm}

Lastly, the following result, known as Van Lambalgen's theorem, will be useful to us (we state the result only for $3^\omega$).  Given measures $\mu$ and $\nu$ on $3^\omega$, we will write $\mu\oplus\nu$ as the measure on $3^\omega$ defined by the following condition: for any string of the form $\sigma\oplus\tau$ for $\sigma, \tau\in 3^{< \omega}$ with $|\sigma|=|\tau|$, $(\mu\oplus\nu)(\sigma\oplus\tau)=\mu(\sigma)\nu(\tau)$.

\begin{thm}[\cite{Van90}]\label{thm-vlt}
Let $\mu$ and $\nu$ be computable measures on $3^\omega$.  Then for $x\oplus y\in 3^\omega$,
$x\oplus y\in\MLR_{\mu\oplus\nu}$ if and only if $x\in\MLR_\mu^y$ and $y\in\MLR_\nu$.
\end{thm}

\subsection{Dimensions of Sequences}

Originally, Lutz defined the \emph{dimension} $\dim(x)$ of a sequence $x \in n^{\omega}$ using a generalized notion of a martingale, called a \emph{gale} \cite{jLutz03a}. This notion of dimension can also be extended to individual points in Euclidean space, and various connections between the dimensions of points and classical Hausdorff dimension have been established. For example, it was shown by Hitchcock \cite{jHitc05} that, for any set $E \subseteq \mathbb{R}^n$ that is a union of $\Pi^0_1$ sets,
\[
\dim_H(E) = \displaystyle\sup_{x \in E}\dim(x),
\]
where $\dim_H(E)$ is the classical Hausdorff dimension of $E$.  Another point-wise characterization of Hausdorff dimension was proven by Lutz and Lutz \cite{jLutz18} and states that, for any set $E \subseteq \mathbb{R}^n$,
\[
\dim_H(E) = \displaystyle\min_{A \subseteq \mathbb{N}} \displaystyle\sup_{x \in E} \dim^A(x),
\]
where $\dim^A(x)$ is the dimension of point $x \in \mathbb{R}^n$ relative to an oracle $A \subseteq \mathbb{N}$. Mayordomo showed that the dimensions of sequences can be characterized using Kolmogorov complexity \cite{jMayo02}, which we briefly discuss below.

A Turing machine $U$ is \emph{universal} if, for all Turing machines $M$ there exists a string $\sigma_M \in 2^{<\omega}$ such that, for all strings $\pi \in 2^{<\omega}$, $U(\langle \sigma_M, \pi \rangle) = M(\pi)$, where $\langle \cdot,\cdot \rangle$ is some string pairing function, e.g., $\langle \sigma,\tau \rangle = 0^{|\sigma|}1\sigma\tau$.

Fix a universal Turing machine $U$. The \emph{Kolmogorov complexity} of a string $\sigma \in n^{<\omega}$ is
\[
C(\sigma) = \min\{|\pi| \, \colon \pi \in 2^{<\omega} \text{ and } U(\pi) = \sigma\}.
\]
There are several ``flavors'' of Kolmogorov complexity. The one described above is referred to as the \emph{plain} Kolmogorov complexity of a string. However, other variants exist such as the \emph{prefix-free} Kolmogorov complexity, which restricts the domain of the Turing machines (including the universal Turing machine) to a prefix-free set. For a detailed discussion on Kolmogorov complexity, see \cite{bLiVit08}. 

The \emph{dimension} of a sequence $x \in 2^{\omega}$ is defined by
\[
\dim(x) = \displaystyle\liminf_{r \rightarrow \infty}\frac{C(x \upharpoonright r)}{r}.
\]
It should be noted that any variation of Kolmogorov complexity can be used in the definition of the dimension of a sequence.

For all sequences $x \in 2^{\omega}$, $0 \leq \dim(x) \leq 1$. If $x \in 2^{\omega}$ is computable, then $\dim(x) = 0$, and if $x$ is algorithmically random, then $dim(x) = 1$. However, there exist sequences with dimension 1 that are not algorithmically random. For any $\alpha \in [0,1]$, there exists a sequence $x \in 2^{\omega}$ such that $\dim(x) = \alpha$ \cite{jLutz03a}.

\subsection{Algorithmically random closed subsets of $\cs$}\label{subsec-rcs}

Recall that $\K(\cs)$ is the collection of all non-empty closed subsets of
$\cs$. Equivalently, these are
the sets of paths through infinite binary trees.  Following \cite{BarBroCen07}, we will code infinite trees by members of $3^{\omega}$.  
Given $x\in 3^{\omega}$, define a tree $T_{x}\subseteq \str$
inductively as follows. First $\epsilon$, the empty string, is
automatically in $T_{x}$. Now suppose $\sigma\in T_{x}$ is the $(i+1)$-st extendible node in $T_x$. Then
\begin{itemize}
  \item ${\sigma}^{\frown} 0 \in T_{x}$ and ${\sigma}^{\frown} 1
  \notin T_{x}$ if $x(i)=0$;
  \item ${\sigma}^{\frown} 0 \notin T_{x}$ and ${\sigma}^{\frown}
  1 \in T_{x}$ if $x(i)=1$;
  \item ${\sigma}^{\frown} 0 \in T_{x}$ and ${\sigma}^{\frown} 1
  \in T_{x}$ if $x(i)=2$.
\end{itemize}
Under this coding $T_{x}$ has no dead ends and hence is always
infinite. Note that every tree without dead ends can be coded by some
$x\in3^\omega$.  We will thus write $\Theta\colon 3^\omega\rightarrow \K(\cs)$ as the map that sends each $x\in 3^\omega$ to the closed set that it codes.

Given a measure $\mu$ on $3^\omega$, we set $\mu^*$ to be the measure on $\K(\cs)$ induced by $\mu$ and $\Theta$, i.e., 
\[
\mu^*(U):= \mu_\Theta(U)=\mu(\Theta^{-1}(U)).
\]
As noted in the introduction, we are particularly interested in certain Bernoulli measures on $3^\omega$.  For $p,q\geq 0$ satisfying $p+q\leq 1$, $\mu_{\langle p,q\rangle}$ is the Bernoulli measure on $3^\omega$ defined by setting, for each $\sigma\in 3^{<\omega}$, 
\begin{itemize}
\item $\mu_{\langle p,q\rangle}(\sigma0\mid\sigma)=p$,
\item $\mu_{\langle p,q\rangle}(\sigma1\mid\sigma)=q$, and
\item $\mu_{\langle p,q\rangle}(\sigma2\mid\sigma)=1-p-q$.
\end{itemize}
In the case that $p=q$, we will write $\mu_{\langle p,q\rangle}$ as $\mu_p$.  We will refer to $\mu_p$ as a \emph{symmetric} Bernoulli measure (as the probabilities of the occurrence of a single branch in the corresponding closed set are equal).

Lastly, for any computable measure $\mu$ on $3^\omega$, we can define a non-empty closed set $C \in \K(\cs)$ to be $\mu^*$-Martin-L\"of random  if $C=[T_{x}]$ for some $x \in \MLR_{\mu}$.

\section{The Converse of the Intersection Theorem}\label{sec-intersection}

In this section, we prove the converse of Theorem \ref{thm-cw1}, thereby answering Question \ref{q2}.  To do so, we provide an alternative proof of Theorem \ref{thm-cw1} in terms of effective Galton-Watson trees, an approach introduced by Diamondstone and Kjos-Hanssen \cite{DiaKjo12} that yields an alternative characterization of random closed sets.

\subsection{Effective Galton-Watson Trees with Two Survival Parameters}\label{subsec-gw}

The idea behind Galton-Watson trees is straightforward:  in the cases considered by Kjos-Hanssen and Diamondstone, we fix some parameter $\beta$, called the \emph{survival parameter}, and we prune $\str$ node by node, leaving a node $\sigma\in\str$ with probability $\beta$, in which case we say that $\sigma$ \emph{survives} (and otherwise we remove it). Note that if $\sigma$ survives, this does not guarantee that infinitely many extensions of $\sigma$ will survive.  As shown by Kjos-Hanssen and Diamondstone, once we have finished pruning $\str$, in the case that we do not have a finite tree, the set of infinite paths through the pruned tree forms a random closed set.  Bienvenu, Hoyrup, and Shen later provided a streamlined approach of the equivalence of these approaches for the case $p=\frac{1}{3}$ (using the machinery of layerwise computability), which can be straightforwardly generalized to arbitrary computable $p$.

In the case of $\mu^*_p$-random closed sets for some $p\in[0,\frac{1}{2}]$, the corresponding Galton-Watson tree that induces the same class of random closed sets is given by using the survival parameter $\beta=1-p$. However, for the case that we consider here, in which the Bernoulli measures on $3^\omega$ need not be symmetric, we need to work with two different survival parameters.  For $i\in\{0,1\}$, we let $\beta_i$ be the probability of survival for any string $\sigma$ that ends with the bit $i$.  As we will see, in the case of the measure $\mu^*_{\langle p,q\rangle}$, we set $\beta_0=1-q$ and $\beta_1=1-p$.  
Due to the condition that $0\leq p+q\leq 1$, we will only consider $\beta_0$ and $\beta_1$ satisfying $1\leq\beta_0+\beta_1\leq 2$.

As an alternative to representing a random closed set in terms of a code for the underlying binary tree with no dead ends, we will represent such a random closed set in terms of the code for a Galton-Watson tree, an approach first used in \cite{CulPor16}.  Whereas the former codes are sequences in $3^\omega$\!, the latter codes will be given by a sequence in $4^{\omega}$\!, where 0s, 1s, and 2s function as they do in the original coding and a $3$ at a given node indicates that the tree is dead above that node. That is, given $x\in
4^{\omega}$\!, we define a tree $S_{x}\subseteq \str$ inductively as
follows. First $\epsilon$, the empty string, is  included in
$S_{x}$ by default. Now suppose that $\sigma\in S_{x}$ is the $(i+1)$-st
surviving node in $S_x$ (i.e., we have yet to determine which, if any, extensions of $\sigma$ are in $S_x$). Then
\begin{itemize}
  \item ${\sigma}0 \in S_{x}$ and ${\sigma}1
  \notin S_{x}$ if $x(i)=0$;
  \item ${\sigma}0 \notin S_{x}$ and ${\sigma}
  1 \in S_{x}$ if $x(i)=1$;
  \item ${\sigma}0 \in S_{x}$ and ${\sigma}1
  \in S_{x}$ if $x(i)=2$;
  \item ${\sigma}0 \notin S_{x}$ and ${\sigma}
  1 \notin S_{x}$ if $x(i)=3$.
\end{itemize}

The above four possibilities correspond to the outcomes of a Galton-Watson tree, where for each non-empty $\sigma\in\str$ we randomly remove $\sigma$ from $\str$ (independently of the other $\tau\in\str$).  The set of infinite paths through the resulting random tree is thus a random closed set.  In fact, as shown by Diamondstone and Kjos-Hanssen, if each edge is removed with probability 
$p$, the resulting distribution on the collection of closed sets is the same as the one given by the measure $\mu^*_p$, with one exception:  the former distribution also includes the empty set as an atom (that is, $\{\emptyset\}$ is given positive measure by the resulting measure on $\K(\cs)$), as there is a non-zero probability that the process of removing edges will produce a finite tree.


We represent this process by a measure as follows:  Let $\nu$ be the measure on $4^{\omega}$ induced by setting, for
each $\sigma\in 4^{<\omega}$\!,
\begin{itemize}
\item $\nu(\sigma0\mid\sigma)=a_0=\beta_0(1-\beta_1)$,
\item $\nu(\sigma1\mid\sigma)=a_1=\beta_1(1-\beta_0)$,
\item $\nu(\sigma2\mid\sigma)=a_2=\beta_0 \beta_1$,
\item $\nu(\sigma3\mid\sigma)=a_3=(1-\beta_0)(1-\beta_1)$.
\end{itemize}
We will refer to $\nu$ as the measure on $4^\omega$ given by survival parameters $(\beta_0,\beta_1)$.

$\nu$ induces a measure on $\Tree$, the space of binary trees.  In this case, the probability of extending
a string in a tree by only $0$ is $a_0$, by only $1$ is $a_1$, by both
$0$ and $1$ is $a_2$, and by neither is $a_3$. Let us say that a tree $T$ is
\emph{GW$(\beta_0,\beta_1)$-random} if it has a $\nu$-Martin-L\"of random code, where $\nu$ is the measure on $4^\omega$ given by survival parameters $(\beta_0,\beta_1)$, i.e.\ if there is some $x\in
\MLR_\nu$ such that $T=S_{x}$.  The result relating GW($\beta_0,\beta_1$)-random trees and random closed sets is the following:

\begin{thm}\label{gwiffrand}
For $\beta_0,\beta_1\in(0,1)$ satisfying $\beta_0+\beta_1\geq 1$, a closed set $C$ is the set of paths through an infinite GW$(\beta_0,\beta_1)$-random tree if and only if it is a $\mu^*_{\langle 1-\beta_1,1-\beta_0\rangle}$-random closed set.
\end{thm}

To prove Theorem \ref{gwiffrand}, we adapt an argument due to Bienvenu, Hoyrup, and Shen \cite{BieHoyShe17}, who, as noted above, give an alternative proof of the Diamondstone/Kjos-Hanssen result  \cite{DiaKjo12} using the machinery of layerwise computability (which we define shortly).  First, since the process that produces a Galton-Watson tree can yield either a finite tree or an infinite tree, we need to determine the probability of each outcome.

\begin{lem}\label{inftree}
The probability that a GW$(\beta_0,\beta_1)$-random tree is infinite is $\dfrac{a_2-a_3}{a_2}=\dfrac{\beta_0+\beta_1-1}{\beta_0\beta_1}$.
\end{lem}

\begin{proof}
Let $r_n$ be the probability that a Galton-Watson tree contains a string of length $n$, which is clearly a non-increasing sequence.  Then we have the following recurrence relation:
\begin{equation*}\tag{$\dagger$}
r_{n+1}=a_0r_n + a_1r_n + a_2(2r_n-r_n^2)
\end{equation*}
The first term corresponds to the case that 0 is the only child of the root, followed by a tree of height $n$, the second term corresponds to the case that 1 is the only child of the root, followed by a tree of height $n$, and the third term corresponds to the case that both 0 and 1 survive and at least one is followed by a tree of height $n$.  

If we take the limit as $n$ approaches infinity of both sides of the above recurrence relation, we get
\[
\ell= (a_0+a_1+2a_2)\ell-a_2\ell^2,
\]
This simplifies to 
\[
a_2\ell^2-(a_0+a_1+2a_2-1)\ell=0.
\]
Since $a_0+a_1+a_2+a_3=1$, this implies
\[
a_2\ell^2-(a_2-a_3)\ell=0.
\]
which has solutions $\ell=0$ and $\ell=\dfrac{a_2-a_3}{a_2}$.  We claim that $r_n>\dfrac{a_2-a_3}{a_2}$ for all $n\in\omega$.  We proceed by induction.  First $r_0=1$ (since we assume that each Galton-Watson tree at least contains the empty string).  Next, assuming that $r_n>\dfrac{a_2-a_3}{a_2}$ for some fixed $n$, suppose for the sake of contradiction that $r_{n+1}\leq \dfrac{a_2-a_3}{a_2}$.  Combining this assumption with ($\dagger$) above yields the inequality
\[
a_2r_n^2-(a_0+a_1+2a_2)r_n+\dfrac{a_2-a_3}{a_2}\geq 0.
\]
Since $a_0+a_1+2a_2=1+a_2-a_3$, the above inequality, after factoring, can be rewritten as
\[
\Bigl(a_2r_n-1\Bigr)\left(r_n-\dfrac{a_2-a_3}{a_2}\right)\geq 0.
\]
This inequality holds precisely when both factors are positive or both factors are negative.  In the former case, using the first factor, we can conclude that $r_n\geq \dfrac{1}{a_2}> 1$ (since, by assumption, $\beta_0,\beta_1\in(0,1)$), which is impossible.  In the latter case, using the second factor, we can conclude that $r_n\leq \dfrac{a_2-a_3}{a_2}$, which contradicts our original assumption about $r_n$.  Thus it follows that $r_{n+1}>\dfrac{a_2-a_3}{a_2}$.  Since $r_n>\dfrac{a_2-a_3}{a_2}$ for all $n$, it follows that $\ell=\lim_{n\rightarrow\infty}r_n=\dfrac{a_2-a_3}{a_2}$.
\end{proof}

Hereafter, let us say that a GW-tree $T$ becomes \emph{extinct} if $T$ is finite, and that a GW-tree $T$ becomes \emph{extinct above} $\sigma\in T$ if $T$ only contains finitely many extensions of $\sigma$.

\begin{lem}\label{lrb}
The probability that any node in a pruned, infinite GW$(\beta_0,\beta_1)$-random tree has two children is $\beta_0 + \beta_1 - 1$, only a left child is $1 - \beta_1$, and only a right child is $1 - \beta_0$.
\end{lem}

\begin{proof}
By Lemma \ref{inftree}, the probability that a GW$(\beta_0,\beta_1)$-random tree does not become extinct is
\[
\dfrac{\beta_0+\beta_1-1}{\beta_0\beta_1}.
\]
Thus, for any string $\sigma\in\str$ in a GW$(\beta_0,\beta_1)$-random tree, the probability that $\sigma$ has two children above both of which the tree does not become extinct is
\[
\beta_0\beta_1 \bigg ( \frac{\beta_0 + \beta_1 - 1}{\beta_0\beta_1} \bigg )^2.
\]
So, the probability that $\sigma$ has two children in a pruned, infinite GW$(\beta_0,\beta_1)$-random tree is 
\begin{align*}
\frac{\beta_0\beta_1 \bigg ( \frac{\beta_0 + \beta_1 - 1}{\beta_0\beta_1} \bigg )^2}{\frac{\beta_0 + \beta_1 - 1}{\beta_0\beta_1}}  &= \beta_0\beta_1 \bigg ( \frac{\beta_0 + \beta_1 - 1}{\beta_0\beta_1} \bigg )\\
&= \beta_0 + \beta_1 - 1.
\end{align*}
The probability that a string $\sigma\in\str$ in a GW$(\beta_0,\beta_1)$-random tree only has a left child above which the tree does not become extinct is
\[
\beta_0 \bigg ( \frac{\beta_0 + \beta_1 - 1}{\beta_0\beta_1} \bigg ) \bigg ( 1 - \beta_1 \bigg ( \frac{\beta_0 + \beta_1 - 1}{\beta_0\beta_1} \bigg ) \bigg ).
\]
Therefore, the probability that $\sigma$ has only a left child in a pruned, infinite GW$(\beta_0,\beta_1)$-random tree is
\begin{align*}
\frac{\beta_0 \bigg ( \frac{\beta_0 + \beta_1 - 1}{\beta_0\beta_1} \bigg ) \bigg ( 1 - \beta_1 \bigg ( \frac{\beta_0 + \beta_1 - 1}{\beta_0\beta_1} \bigg ) \bigg )}{\frac{\beta_0 + \beta_1 - 1}{\beta_0\beta_1}} &= \frac{\beta_0 \bigg ( \frac{\beta_0 + \beta_1 - 1}{\beta_0\beta_1} \bigg ) - \beta_0\beta_1 \bigg ( \frac{\beta_0 + \beta_1 - 1}{\beta_0\beta_1} \bigg )^2}{\frac{\beta_0 + \beta_1 - 1}{\beta_0\beta_1}}\\
&= \beta_0 - \beta_0\beta_1 \bigg (\frac{\beta_0 + \beta_1 - 1}{\beta_0\beta_1} \bigg )\\
&= 1 - \beta_1.
\end{align*}
Using a similar argument, we can prove that the probability that $\sigma$ has only a right child in a pruned, infinite GW$(\beta_0,\beta_1)$-random tree is
\[
1 - \beta_0. \qedhere
\]
\end{proof}

From Lemma \ref{lrb}, we see that if we take an infinite GW($\beta_0,\beta_1$)-random tree and remove all of its terminal nodes, then the resulting distribution is given by the measure $\mu^*_{\langle 1-\beta_1,1-\beta_0\rangle}$.  In order to derive Theorem \ref{gwiffrand}, we need to verify that there is an effective procedure that maps a code of an infinite GW($\beta_0,\beta_1$)-random tree to a $\mu^*_{\langle 1-\beta_1,1-\beta_0\rangle}$-random closed set.

In the case of a single parameter GW-tree, this was shown by Bienvenu, Hoyrup, and Shen \cite{BieHoyShe17} using the machinery of layerwise computability, originally defined by Hoyrup and Rojas \cite{HoyRoj09a, HoyRoj09b}.  As defined in \cite{BieHoyShe17}, for a computable measure $\mu$, a mapping $\Phi$ is $\mu$-\emph{layerwise computable} if there is a $\mu$-Martin-L\"of test $(U_i)_{i \in \omega}$ and a Turing machine $M$ such that, for any $n \in \omega$ and $x \notin U_n$, $M(n,x) = \Phi(x)$ (here we think of $M$ as being equipped with an oracle tape on which $x$ is written and a second tape containing the input $n$). Intuitively, the lemma below shows that there exists a layerwise computable mapping that converts a code for an infinite tree with dead ends to a code for the same infinite tree with the dead ends removed.

\begin{lem}\label{lwc}
There exists a $\nu$-layerwise computable mapping $\Phi: 4^{\omega} \rightarrow 4^{\omega}$ such that, for all $x \in 4^{\omega}$, if $\Phi(x)\in 3^\omega$, $T_{\Phi(x)}$ is the set of all infinite branches of $T_x$.
\end{lem}

\begin{proof}
For all $\sigma \in 2^{<\omega}$ and $n \in \omega$, let
\[
U^{\sigma}_n = \{x \in 4^{\omega} \colon (\exists \tau \succeq \sigma) \tau \in T_x\uh(|\sigma| + n) \;\&\; (\exists k > n)(\forall \tau \succeq \sigma) \tau \notin T_x \uh (|\sigma| + k)\},
\]
where $T \upharpoonright i$ is the set of all strings in $T$ of length $i \in \omega$, for any tree $T$. That is, $U^{\sigma}_n$ consists of all codes $x \in 4^{\omega}$ such that a length $n$ extension of $\sigma$ survives in $T_x$ but  $T_x$ eventually becomes extinct above $\sigma$.  Clearly, $(U^\sigma_n)_{n\in\omega}$ is effectively open uniformly in $\sigma$.

Using the notation from Lemma \ref{inftree}, where $r_n$ is the probability that GW-tree contains a string of length $n$ and $\ell$ is the probability that such a tree contains a string of every length, then, for all $\sigma \in 2^{<\omega}$ and $n \in \omega$,
\begin{align*}
\nu(U^{\sigma}_n) = r_n - \ell,
\end{align*}
where $\nu$ is the measure on $4^\omega$ from the definition of a random GW($\beta_0,\beta_1$)-tree.  Observe that, for any $\sigma \in \str$, as $n$ increases $\nu(U^{\sigma}_n)$ decreases and approaches zero (effectively in $n$). Therefore, there is a computable subsequence of indices $(n_i)_{i \in \omega}$ such that, for all $i \in \omega$,
\begin{align*}
\nu(U^{\sigma}_{n_i}) \leq 2^{-i}.
\end{align*}
 Without loss of generality, by taking an appropriate subsequence and relabeling the indices, we can assume that $\nu(U^{\sigma}_n)\leq 2^{-n}$ for all $i,n\in\omega$.  Thus $(U^{\sigma}_n)_{n\in\omega}$ is a $\nu$-Martin-L\"of test.


Now, letting $(\sigma_i)_{i\in\omega}$ be the enumeration of $\str$ in length-lexicographic order, for all $i \in \omega$, we set
\[
V_i = \bigcup_{j \in \omega}U^{\sigma_j}_{i+j+1}.
\]
Since, $(U^{\sigma_j}_n)_{n \in \omega}$ is a $\nu$-Martin-L\"of test for each $j\in\omega$, we have
\begin{align*}
\nu(V_i) &= \displaystyle\sum_{j \in \omega}\nu(U^{\sigma_j}_{i+j+1})\leq \displaystyle\sum_{m = 1}^{\infty}2^{-(i+m)}= 2^{-i},
\end{align*}
which implies that $(V_i)_{i \in \omega}$ is a $\nu$-Martin-L\"of test.

We now show that $\Phi$ is layerwise computable by describing a Turing machine such that, when given $i \in \omega$ and $x \in 4^{\omega}$ such that $x \notin V_i$, produces $\Phi(x)$ on the output tape.

First, observe that $x\notin V_i$ implies that $x\notin U^{\sigma_j}_{i+j+1}$ for all $j\in\omega$.  In particular, this implies that for each $j\in\omega$, if there is some $\tau\succeq\sigma_j$ such that $\tau\in T_x\uh(|\sigma_j|+i+j+1)$ then for all $k>i+j+1$ there is some $\tau\succeq\sigma_j$ such that $\tau\in T_x\uh(|\sigma_j|+k)$.  In other words, if we see that $T_x$ contains an extension of $\sigma_j$ of length $i+j+1$, then we can conclude that $T_x$ does not become extinct above $\sigma_j$.

Our machine $M$ works by adding binary strings to a set $S$, which is a set of strings above which our procedure will take action (defined below); the set $S$, defined in stages, is equal to the set of extendible nodes of $T_x$.   First, $M$ sets $S_0
 = \emptyset$ and constructs the tree coded by $x$ to see if $T_x$ contains a string of length $i + 1$. If so, then by the discussion in the previous paragraph, $T_x$ contains a string of every length and thus is infinite, and then $M$ places the empty string $\epsilon$ inside $S_1$. Otherwise, $S=S_0=\emptyset$, $M$ halts, and $\Phi(x)$ outputs $3^\infty$.

If $S_1\neq\emptyset$, we proceed inductively as follows.  For $k\geq 1$, assume that $M$ has already produced $k-1$ bits of output and suppose that $\sigma\in S_k$ is the lexicographically least string above which we have not taken action.  We describe how our procedure takes action above $\sigma$.
Since $\sigma=\sigma_j$ for some $j\in\omega$, the two extensions of $\sigma_j$ are $\sigma_{2j+1}=\sigma_j0$ and $\sigma_{2j+2}=\sigma_j1$ (as we are using the standard length-lexicographic ordering of $\str$).  $M$ then constructs sufficiently many levels of the tree $T_x$ to determine if $T_x$ is extinct $i + 2j + 2$ levels above  $\sigma_j0$.  If not, we enumerate $\sigma_j0$ into $S_{k+1}$.  Then $M$ similarly determines whether $T_x$ is extinct $i+2j+3$ levels above $\sigma_j1$; in the case that it is not, we also enumerate $\sigma_j1$ into $S_{k+1}$.  Thus, either $\sigma_j0$ or $\sigma_j1$ are added to $S_{k+1}$, or both.

If only $\sigma_j0$ is added to $S_{k+1}$, then $M$ outputs 0. If only $\sigma_j1$ was added to $S_{k+1}$, then $M$ outputs 1. If both $\sigma_j0$ and $\sigma_j1$ were added, then $M$ outputs 2.  It is straightforward to verify that $\Phi(x)$ is the desired layerwise computable mapping. \qedhere
\end{proof}

Note that given any $\nu$-random $x\in 4^\omega$ that codes for a tree with no infinite paths, we have $\Phi(x)=3^\infty$.  By Lemma \ref{inftree}, $3^\infty$ is an atom of the measure induced by $\Phi$ and $\nu$; in fact, the singleton $\{3^\infty\}$ is given measure $1-\frac{\beta_0+\beta_1-1}{\beta_0\beta_1}$.  Moreover, by Lemma \ref{lrb}, the measure on $3^\omega$ induced by $\Phi$ when restricted to those $x\in 4^\omega$ that code for a tree with infinite paths (obtained by considering the range of $\Phi$ without the sequence $3^\infty$ and scaling the measure appropriately) is precisely the measure $\mu^*_{\langle 1-\beta_1,1-\beta_0\rangle}$.
We can thus conclude the proof of Theorem \ref{gwiffrand} using the fact that both randomness preservation and no randomness from non-randomness hold for layerwise computable maps (see \cite[Proposition 5.3.1]{HoyRoj09b}):  randomness preservation ensures that $\Phi$ maps an infinite GW$(\beta_0,\beta_1)$-random tree $T$ to the corresponding $\mu^*_{\langle 1-\beta_1,1-\beta_0\rangle}$-random closed set $[T]$, and no randomness from non-randomness ensures that every $\mu^*_{\langle 1-\beta_1,1-\beta_0\rangle}$-random closed set $C$ is the image of some infinite GW$(\beta_0,\beta_1)$-random tree $T$ under $\Phi$ with $C=[T]$.

\subsection{Intersections}

We now turn to the main result of the section, which provides an affirmative answer to Question \ref{q2}.  Here the machinery we laid out in the previous section will prove to be useful.

\begin{thm}\label{thm-conv}
Suppose that $p, q, r, s\geq 0$, $0\leq p + q\leq 1$ and $0\leq r + s\leq1$.  If $R\in\K(\cs)$ is Martin-L\"of random with respect to the measure $\mu^*_{\langle p+r-pr,q+s-qs \rangle}$, then there are 
$P,Q\in \K(\cs)$ such that
\begin{itemize}
\item[(i)] $P$ is $\mu^*_{\langle p,q\rangle}$-random relative to $Q$,
\item[(ii)] $Q$ is  $\mu^*_{\langle r,s\rangle}$-random relative to $P$, and
\item[(iii)] $R=P\cap Q$.
\end{itemize}
\end{thm}

To prove Theorem \ref{thm-conv}, we will reprove part (3) of the Intersection Theorem using the lemma below, from which the converse will immediately follow by no randomness from non-randomness.

\begin{lem}\label{fit}
Suppose that $p, q, r, s \geq 0$, $0 \leq p + q \leq 1$ and $0 \leq r + s \leq 1$, and let $P \in \K(\cs)$ be $\mu^*_{\langle p,q \rangle}$-random relative to $Q \in \K(\cs)$ and $Q$ be $\mu^*_{\langle r,s \rangle}$-random relative to $P$. There exists a layerwise computable mapping $\Gamma: 3^{\omega} \rightarrow 4^{\omega}$ such that, if $P \cap Q \neq \emptyset$ and $x_P, x_Q \in 3^{\omega}$  are codes for $P$ and $Q$, respectively, then $\Gamma(x_P \oplus x_Q)\in 3^\omega$ is a code for $P \cap Q$ and is $\mu_{\langle p+r-pr, q+s-qs \rangle}$-random.
\end{lem}

\begin{proof}
First we describe a total computable mapping $\Psi: 3^{\omega} \rightarrow 4^{\omega}$ such that, on input $x = y \oplus z$, $\Psi$ produces a code for the tree $T_y \cap T_z$, which may include non-extendible nodes (and may even be finite). We define a machine $M$ corresponding to this mapping as follows. 


On input $y\oplus z$, $M$ yields its output on the basis of an enumeration of $T_y\cap T_z$, which we shall write as $T$.  First, $M$ places $\epsilon$ into $T$.  Next, $M$ enumerates $T$ level by level as follows.  Suppose that $T$ has been defined for all strings of length $\ell$.  For each $\sigma\in T$ of length $\ell$, taken in lexicographic order,  $M$ checks to see if $\sigma0$ and $\sigma1$ are also in $T_y\cap T_z$ (using the input $y\oplus z$).  There are four cases to consider:
\begin{itemize}
\item Case 1: $\sigma0\in T_y\cap T_z$ and $\sigma1\notin T_y\cap T_z$, in which case $\sigma0$ is placed into $T$ and $M$ outputs a 0.
\item Case 2: $\sigma0\notin T_y\cap T_z$ and $\sigma1\in T_y\cap T_z$, in which case $\sigma1$ is placed into $T$ and $M$ outputs a 1.
\item Case 3: $\sigma0\in T_y\cap T_z$ and $\sigma1\in T_y\cap T_z$, in which case both $\sigma0$ and $\sigma1$ are placed into $T$ and $M$ outputs a 2.
\item Case 4: $\sigma0\notin T_y\cap T_z$ and $\sigma1\notin T_y\cap T_z$, in which case neither $\sigma0$ nor $\sigma1$ is placed into $T$ and $M$ outputs a 3.
\end{itemize}

\noindent If at any point during this process, there are no new strings for $M$ to add to $T$, $T_y \cap T_z$ is a finite tree and $M$ outputs an infinite sequence of 3's for its remaining output.  Lastly, let $\nu$ be the measure on $3^\omega$ induced by intersecting a $\mu^*_{\langle p,q \rangle}$-random closed set $P$ with a $\mu^*_{\langle r,s \rangle}$-random closed set $Q$ that is random relative to $P$; that is, $\nu=(\mu_{\langle p,q \rangle}\oplus \mu_{\langle r,s \rangle})\circ \Psi^{-1}$ (we will explicitly calculate this measure below).


By Lemma \ref{lwc}, there exists a $\nu$-layerwise computable mapping $\Phi: 4^{\omega} \rightarrow 4^{\omega}$ such that, for all $x \in 4^{\omega}$, if $\Phi(x)\in 3^\omega$, $T_{\Phi(x)}$ is the set of all infinite branches of $T_x$. This means that there exists a $\nu$-Martin-L\"of test $(V_i)_{i \in \omega}$ (as in the proof of Lemma \ref{lwc}) and a Turing machine $M'$ such that $M'(x,i) = \Phi(x)$ for any $i \in \omega$ and $x \notin V_i$.  We would like to compose $\Psi$ with $\Phi$ to define $\Gamma$, but some care is needed.

For each $n\in\omega$, define $W_n=\Psi^{-1}(V_n)$.  We verify that $(W_n)_{n\in\omega}$ is a Martin-L\"of test with respect to the measure  $\mu_{\langle p,q \rangle}\oplus \mu_{\langle r,s \rangle}$.   First, since $\Psi$ is total, $(W_n)_{n\in\omega}$ is uniformly $\Sigma^0_1$.  Moreover, we have, for each $n\in\omega$,
\[
(\mu_{\langle p,q \rangle}\oplus \mu_{\langle r,s \rangle})(W_n)=(\mu_{\langle p,q \rangle}\oplus \mu_{\langle r,s \rangle})(\Psi^{-1}(V_n))=\nu(V_n)\leq 2^{-n}.
\]

Now, setting $\Gamma=\Phi\circ\Psi$, we claim that $\Gamma:3^\omega\rightarrow 4^\omega$ is the desired $(\mu_{\langle p,q \rangle}\oplus \mu_{\langle r,s \rangle})$-layerwise computable map, defined in terms of the test $(W_n)_{n\in\omega}$.  Given $P,Q\in\mathcal{K}(\cs)$ satisfying the hypothesis of the theorem, since $P$ $\cap$ $Q \neq \emptyset$, there exists some $n \in \omega$ such that $\Psi(x_P \oplus x_Q) \notin V_n$ by randomness preservation.  Thus $x_P \oplus x_Q\in\Psi^{-1}(V_n)$, and so 
\[
\Gamma(x_P \oplus x_Q)=\Phi(\Psi(x_P \oplus x_Q))
\]
yields the code of $P\cap Q$ as an element of $3^\omega$. (Note that $\Gamma(x_P \oplus x_Q)\notin 3^\omega$ if and only if $\Gamma(x_P \oplus x_Q)=3^\infty$ if and only if $P\cap Q=\emptyset$.)

Finally, we ensure that $\Gamma$ induces the measure $\mu_{\langle p+r-pr, q+s-qs \rangle}$ on $3^\omega$ (once we ignore the atom $3^\infty$ and scale the induced measure as in the proof of Theorem \ref{gwiffrand}). Given $x_P\oplus x_Q$ as above, let $y = \Psi(x_P \oplus x_Q) \in 4^{\omega}$. Under the assumption that $P \cap Q \neq \emptyset$, we calculate the probabilities that $y(i)$ is a 0, 1, 2, or 3, for any $i \in \omega$:
\begin{itemize}
\item $y(i) = 0$ if $(x_P(0), x_Q(0)) \in \{(0,0), (0,2), (2,0)\}$, which occurs with probability $pr + p(1 - r - s) + r(1 - p - q)$,
\item $y(i) = 1$ if $(x_P(0), x_Q(0)) \in \{(1,1), (1,2), (2,1)\}$, which occurs with probability $qs + q(1 - r - s) + s(1 - p - q)$,
\item $y(i) = 2$ if $(x_P(0), x_Q(0)) \in \{(2,2)\}$, which occurs with probability $(1 - p - q)(1 - r - s)$, and
\item $y(i) = 3$ if $(x_P(0), x_Q(0)) \in \{(0,1), (1,0)\}$, which occurs with probability $ps + qr$.
\end{itemize}
Therefore, the survival parameters of the resulting GW$(\beta_0,\beta_1)$-random tree are

\begin{align*}
\beta_0 &= pr + p(1 - r - s) + r(1 - p - q) + (1 - p - q)(1 - r - s)\\
				&= p(1 - s) + (1 - p - q)(1 - s)\\
				&= (1 - s)(1 - q)\\
				&= 1 - q - s + qs.
\end{align*}

and

\begin{align*}
\beta_1 &= qs + q(1 - r - s) + s(1 - p - q) + (1 - p - q)(1 - r - s)\\
				&= q(1 - r) + (1 - p - q)(1 - r)\\
				&= (1 - r)(1 - p)\\
				&= 1 - p - r + pr
\end{align*}
The code $\Gamma(x_P \oplus x_Q)\in 3^\omega$ represents a pruned, infinite GW$(\beta_0,\beta_1)$-random tree. By Lemma \ref{lrb}, the probability that any node in the  tree encoded by $\Gamma(x_P \oplus x_Q)$ has only a left child is $1-\beta_0=p + r - pr$, only a right child is $1-\beta_1=q + s - qs$, and both children is
\begin{align*}
\beta_0 + \beta_1 - 1 &= 1 - q - s + qs + 1 - p - r + pr - 1\\
											&= 1 - (p + r - pr) - (q + s - qs).
\end{align*}
Therefore, $\Gamma(x_P \oplus x_Q)$ is $\mu_{\langle p+r-pr, q+s-qs \rangle}$-random.
\end{proof}

Part (3) of the Intersection Theorem follows directly from the lemma above and the fact that randomness preservation holds for layerwise computable mappings. Finally, Theorem \ref{thm-conv} follows directly by an application of the no-randomness-from-nonrandomness principle.

 \section{Multiple Intersections of Random Closed Sets}\label{sec-multiple}

In the case that we are dealing with two closed sets that are random with respect to the same symmetric Bernoulli measures, i.e. $p=q=r=s$, the key inequality $p+q+r+s<1+pr+qs$ in the Intersection Theorem becomes $4p<1+2p^2$.  Since $2p^2-4p-1=0$ has solutions $p= 1\pm\frac{\sqrt{2}}{2}$, and we are only interested in the case $p=1-\frac{\sqrt{2}}{2}$ (since $1+\frac{\sqrt{2}}{2}>1$), this key inequality is equivalent to the condition $p<1-\frac{\sqrt{2}}{2}$.  This allows us to derive Corollary \ref{cor-cw}, which we restate here for the sake of convenience:

\bigskip

\noindent {\bf Corollary 1}\;(Cenzer/Weber \cite{CenWeb13}){\bf .}
For $p\in[0,1/2]$, let $P,Q\in\K(\cs)$ be relatively $\mu^*_p$-random.
\begin{enumerate}
\item If $p\geq 1-\frac{\sqrt{2}}{2}$, then $P\cap Q=\emptyset$.
\item If $p< 1-\frac{\sqrt{2}}{2}$, then $P\cap Q=\emptyset$ with probability $\frac{2p^2}{(1-2p)^2}$.
\item If $p< 1-\frac{\sqrt{2}}{2}$ and $P\cap Q\neq\emptyset$, then $P\cap Q$ is Martin-L\"of random with respect to the measure $\mu^*_{2p-p^2}$.
\end{enumerate}

We would like to extend this analysis to determine which parameters $p$ allow for the possibility that $n$ $\mu_p^*$-random closed sets have a non-empty intersection for various choices of $n\in\omega$.  Here we need to be more precise: let us say that closed sets $P_1,\dotsc,P_n$ are \emph{mutually $\mu^*_p$-random} if, setting $y_i=\bigoplus_{j\neq i}x_{P_j}$, we have $x_{P_i}\in\MLR^{y_i}_{\mu_p}$.

In order to state our result, we define a sequence of polynomials $(f_n(p))_{n\geq 1}$  by setting $f_n(p)=1 - (1-p)^n$ for $p\in[0,\frac{1}{2}]$.  The desired generalization can thus be stated as follows:

\begin{thm}\label{thm-multint}
For $p\in [0,\frac{1}{2}]$ and $n\geq 2$, given $n$ mutually $\mu^*_p$-random closed sets $P_1,\dotsc,P_n$, the following hold:
\begin{enumerate}
\item If $p\geq 1-\frac{1}{\sqrt[n]{2}}$, then $\bigcap_{i=1}^n P_i=\emptyset$.
\item If $p< 1-\frac{1}{\sqrt[n]{2}}$, then $\bigcap_{i=1}^n P_i=\emptyset$ with probability $1-\frac{1-2f_n(p)}{(1-2p)^n}$.
\item If $p< 1-\frac{1}{\sqrt[n]{2}}$ and $\bigcap_{i=1}^n P_i\neq\emptyset$, then $\bigcap_{i=1}^n P_i$ is Martin-L\"of random with respect to the measure $\mu^*_{f_n(p)}$.
\end{enumerate}
\end{thm}

%
%

%

In order to prove Theorem \ref{thm-multint}, we make use of several lemmas:
\begin{lem} \label{lem-poly1}
For $p\in[0,\frac{1}{2}]$, the following recursive relation holds:
\begin{itemize}
\item $f_1(p)=p$;
\item $f_{n+1}(p) = p + f_n(p) - p f_n(p)$.
\end{itemize}
\end{lem}

\begin{proof} This follows immediately by induction on $n\in\omega$.
\end{proof}

\begin{lem}\label{lem-poly2}
For $n\geq 1$ and $p\in[0,\frac{1}{2}]$, $f_n(p)<\frac{1}{2}$ if and only if $p< 1-\frac{1}{\sqrt[n]{2}}$.
\end{lem}

\begin{proof}
As $f_n(p)=1-(1-p)^n$, it is straightforward to verify that
\[
1-(1-p)^n<\frac{1}{2} \;\text{ if and only if }\; p < 1-\frac{1}{\sqrt[n]{2}}.\qedhere
\]
\end{proof}

\begin{proof}[Proof of Theorem \ref{thm-multint}]
We proceed via induction on  $n\geq 2$.  For the case $n=2$, noting that $f_1(p)=p$ and $f_2(p)=2p-p^2$, this case is established by Corollary \ref{cor-cw}.  In particular, it is straightforward to verify for part (2) that 
\[
\dfrac{2p^2}{(1-2p)^2}=1-\dfrac{1-2f_2(p)}{(1-2p)^2}.
\]

Assuming the result holds for $n\geq 2$, we consider $n+1$ mutually $\mu^*_p$-random closed sets $P_1,\dotsc,P_{n+1}$.  To verify (1), suppose that $p\geq 1-\frac{1}{\sqrt[n+1]{2}}$.  Assume for the sake of contradiction that $\bigcap_{i=1}^{n+1} P_i\neq\emptyset$.  Then by the inductive hypothesis, $\bigcap_{i=1}^n P_i$ is a $\mu^*_{f_n(p)}$-random closed set.  Then in order to apply
part (1) of the Intersection Theorem to $\bigcap_{i=1}^{n} P_i$ and $P_{n+1}$, where $p=q$ and $r=s=f_n(p)$, we must have 
\[
2p +2f_n(p)< 1+2pf_n(p),
\]
or equivalently, by Lemma \ref{lem-poly1},
\[
f_{n+1}(p)=p+f_n(p)-pf_n(p)< \frac{1}{2}.
\]
By Lemma \ref{lem-poly2}, this implies that $p< 1-\frac{1}{\sqrt[n+1]{2}}$, which contradicts our earlier assumption.  Thus we must have $\bigcap_{i=1}^{n+1} P_i=\emptyset$.

To verify (2), suppose that $p<1-\frac{1}{\sqrt[n+1]{2}}$.  By Lemma \ref{lem-poly2}, this implies that $f_{n+1}(p)<\frac{1}{2}$.  As we have argued in the verification of (1), this in turn implies that $2p+2f_n(p)<1+2pf_n(p)$, and so we can apply part (2) of the Intersection Theorem. By our inductive hypothesis, $\bigcap_{i=1}^{n} P_i\neq\emptyset$ with probability $\frac{1-2f_n(p)}{(1-2p)^n}$ and, in the case that $\bigcap_{i=1}^{n} P_i\neq\emptyset$, it is $\mu^*_{f_n(p)}$-random.  Conditional on the assumption that $\bigcap_{i=1}^{n} P_i\neq\emptyset$, by part (2) of the Intersection Theorem we have $\bigcap_{i=1}^{n+1} P_i\neq\emptyset$ with probability
\[
1-\dfrac{2pf_n(p)}{(1-2p)(1-2f_n(p))}.
\]
Thus, the probability that $\bigcap_{i=1}^{n+1} P_i\neq\emptyset$ is
\[
\dfrac{1-2f_n(p)}{(1-2p)^n}\left(1-\dfrac{2pf_n(p)}{(1-2p)(1-2f_n(p))}\right)=\dfrac{1-2f_n(p)-2p+4pf_n(p)}{(1-2p)^{n+1}}=\dfrac{1-2f_{n+1}(p)}{(1-2p)^{n+1}},
\]
where the last equality follows from Lemma \ref{lem-poly1}.  (2) immediately follows.

To verify (3), suppose that $\bigcap_{i=1}^{n+1} P_i\neq\emptyset$.  Since  $\bigcap_{i=1}^{n} P_i\neq\emptyset$, it follows from the induction hypothesis that $\bigcap_{i=1}^{n} P_i$ is $\mu^*_p$-random.  Then applying (3) of the Intersection Theorem to the case $p=q$ and $r=s=f_n(p)$, $\bigcap_{i=1}^{n+1} P_i$ is random with respect to the measure $\mu^*_{p+f_n(p)-pf_n(p)}$, which, by Lemma \ref{lem-poly1}, is the measure $\mu^*_{f_{n+1}(p)}$.  
\end{proof}

%
%

For $n\geq 1$, since $f_n(0)=0$ and $f_n(\frac{1}{2})=1-(\frac{1}{2})^n$, and $f'_n(p)=n(1-p)^{n-1}$,  $f_n\colon [0,\frac{1}{2}]\rightarrow[0,1-(\frac{1}{2})^n]$ is strictly increasing.  Thus we can define $f^{-1}_n(p)=1-\sqrt[n]{1-p}$ on $[0,1-(\frac{1}{2})^n]$.  Using the functions $f_n^{-1}$, we can obtain a converse to part (3) of Theorem \ref{thm-multint}.  

There is, however, one additional wrinkle.  We would like to prove the result by induction on the number of closed sets in the desired intersection.  To do so, we need an additional hypothesis about the relative randomness of the closed sets over which we will take the intersection.  To prove the unrelativized version of our result, we use a relativized version of the result in the inductive step.  

For $z\in 3^\omega$, let us say that closed sets $P_1,\dotsc, P_n$ are \emph{mutually $\mu^*_p$-random relative to $z$} if, for each $i\in\{1,\dotsc,n\}$, setting
$y_i=\bigoplus_{j\neq i}x_{P_j}$, we have $x_{P_i}\in\MLR^{y_i\oplus z}_{\mu_p}$.  We also make use of the following consequence of van Lambalgen's theorem:  If $x_1,x_2,\dotsc,x_n$ are mutually $\mu_p$-random relative to $z$ and $z$ is $\mu_q$-random, then $z$ is $\mu_q$-random relative to $\bigoplus_{i=1}^nx_i$.  Lastly, we will make use of the following relativized version of Theorem \ref{thm-conv}, which follows from a direct relativization of the proof of Theorem \ref{thm-conv}:  

\begin{thm}\label{thm-intrel}
Suppose that $p, q, r, s\geq 0$, $0\leq p + q\leq 1$ and $0\leq r + s\leq1$.  If $R\in\K(\cs)$ is $\mu^*_{\langle p+r-pr,q+s-qs \rangle}$-Martin-L\"of random relative to $z\in 3^\omega$, then there are 
$P,Q\in \K(\cs)$ such that
\begin{itemize}
\item[(i)] $P$ is $\mu^*_{\langle p,q\rangle}$-random relative to $x_Q\oplus z$,
\item[(ii)] $Q$ is  $\mu^*_{\langle r,s\rangle}$-random relative to $x_P\oplus z$, and
\item[(iii)] $R=P\cap Q$.
\end{itemize}
\end{thm}

Now we state our partial converse of Theorem \ref{thm-multint}:

\begin{thm}
For $p\in[0,\frac{1}{2}]$, suppose that $Q\in\K(\cs)$ is $\mu^*_p$-random relative to $z\in 3^\omega$.  Then for $n\geq 2$, there are $P_1,\dotsc,P_n\in\K(\cs)$ that  are $\mu^*_{f^{-1}_n(p)}$-random relative to $z$  such that  $Q=\bigcap_{i=1}^nP_i$
\end{thm}

\begin{proof}
Again we proceed by induction.  For $n=2$, this follows from Theorem \ref{thm-intrel}.  Suppose now that the result holds for a fixed $n\geq 2$ and all $z\in\cs$.  Let $q=f^{-1}_{n+1}(p)$, so that $f_{n+1}(q)=p$.  In particular, by Lemma \ref{lem-poly1} we have $p=f_{n+1}(q)=q+f_n(q)-qf_n(q)$.  By Theorem \ref{thm-conv}, there are $P_1\in\K(\cs)$ and $R\in\K(\cs)$ such that $Q=P_1\cap R$, $P_1$ is $\mu^*_q$-random relative to $x_R$, and $R$ is $\mu^*_{f_n(q)}$-random relative to $x_{P_1}$.  By the inductive hypothesis, since $R$ is $\mu^*_{f_n(q)}$-random and $f^{-1}_n(f_n(q))=q$, there are  $P_2,\dotsc,P_{n+1}\in\K(\cs)$ that are mutually $\mu^*_q$-random relative to $x_{P_1}$ such that $R=\bigcap_{i=2}^{n+1}P_i$.  By the consequence of van Lambalgen's theorem discussed above, $x_{P_1}$ is $\mu_q$-random relative to $\bigoplus_{i=2}^{n+1}x_{P_i}$, and hence the sequence $P_1,P_2,\dotsc,P_{n+1}$ is mutually $\mu^*_q$-random.  Moreover, $Q=P_1\cap R=\bigcap_{i=1}^{n+1}P_i$, which yields the desired conclusion.
\end{proof}

\section{Multiple Intersections and Effective Dimension}\label{sec-dimension}

In this final section, we relate the effective dimension of a sequence $x$ with what we call the degree of intersectability of a family of random closed sets, at least one of which contains $x$, drawing upon a result of Diamondstone and Kjos-Hanssen  \cite{DiaKjo12} on the relationship between effective dimension and members of random closed sets.

Recall from our discussion at the beginning of Section \ref{subsec-gw} that Diamondstone and Kjos-Hanssen studied the collection of random closed sets obtained as the set of paths through certain Galton-Watson trees.  For $\gamma\in[0,1)$, they further defined the relation ``$x$ \text{is a} {\sc Member}$_\gamma$'' for $x\in\cs$ to mean that $x$ is in a random closed set corresponding to the survival parameter $2^{-\gamma}$ (that is, each node of the full binary tree is removed with probability $1-2^{-\gamma}$).  Moreover, they proved:

\begin{thm}[Diamondstone, Kjos-Hanssen \cite{DiaKjo12}]\label{thm-dkh}
For $x\in\cs$,
\[
\dim(x)>\gamma \;\Rightarrow\; x \text{ is a {\sc Member}$_\gamma$} \;\Rightarrow\; \dim(x)\geq\gamma.
\]
\end{thm}

How does this relate to our present work? As discussed in Section \ref{subsec-gw} above, a random GW-tree with survival parameter $2^{-\gamma}$ is a $\mu_p^*$-random closed set with $p=1-2^{-\gamma}$.  Solving for $\gamma$, we get $\gamma=-\log(1 -p)$.  From Theorem \ref{thm-dkh} we can immediately conclude the following:
\begin{cor}\label{cor-dim}
Let $x\in\cs$.

\medskip

\begin{enumerate}
\item For $p\in[0,\frac{1}{2}]$,
\[
\dim(x)>-\log(1-p) \;\Rightarrow\; x\in P \text{ for some } P\in\MLR_{\mu^*_p} \;\Rightarrow\; \dim(x)\geq-\log(1-p).
\]
\item In the case that $p=1-\frac{1}{\sqrt[n]{2}}$, we have
\[
\dim(x)>\dfrac{1}{n} \;\Rightarrow\; x\in P \text{ for some } P\in\MLR_{\mu^*_p} \;\Rightarrow\; \dim(x)\geq\dfrac{1}{n}.
\]
\end{enumerate}
\end{cor}
\noindent Note that (2) follows immediately from (1), since in the case that $p=1-\frac{1}{\sqrt[n]{2}}$, we have 
\[
-\log(1-p)=\dfrac{1}{n}.
\]

We now relate this to our results on multiple intersections from the previous section.  Let us define the \emph{degree of intersectability} of the family of $\mu_p^*$-random closed sets for a fixed $p\in[0,1/2]$ to be the unique $n$ such that (i) $n$ mutually $\mu_p^*$-random closed sets can have a non-empty intersection and (ii) $n+1$ mutually $\mu_p^*$-random closed sets always have an empty intersection (we know such an $n$ exists by Theorem \ref{thm-multint}).\footnote{We count intersectability in terms of the number of mutually random closed sets that can yield a non-empty intersection, not the number of times we can apply an intersection to a collection of mutually random closed sets to yield a non-empty set.  On the latter approach, in all of the results that follow, the value $n$ would need to be replaced with $n+1$.}  Note that for a family of random closed sets to have degree of intersectability equal to 1, this means that no pair of relatively random closed sets from the family have a non-empty intersection.  The following lemma can be derived from Corollary \ref{cor-cw} and Theorem \ref{thm-multint}.

\begin{lem}{\ }\label{lem-degint} Let $p\in[0,1/2]$.
\begin{enumerate}
\item The family of $\mu_p^*$-random closed sets has degree of intersectability equal to 1 if and only if  $p\in[1 - \frac{1}{\sqrt{2}}, \frac{1}{2}]$.  
\item The family of $\mu_p^*$-random closed sets has degree of intersectability equal to $n\geq 2$ if and only if
\[
p\in\Biggl[1-\dfrac{1}{\sqrt[n+1]{2}},1-\dfrac{1}{\sqrt[n]{2}}\Biggr).
\]
\end{enumerate}
\end{lem}

The above observations allow us to prove the following:

\begin{thm}\label{thm-int-dim}
If $P$ is a symmetric Bernoulli random closed set from a family with degree of intersectability $n$, then for every $x\in P$, we have $\dim(x)\geq\frac{1}{n+1}$.
\end{thm}

\begin{proof}
Let $P$ be a symmetric Bernoulli random closed set from a family with degree of intersectability $n\geq 2$ and let $x\in P$.  By Lemma \ref{lem-degint}(2), $P$ is $\mu^*_p$-random for some 
\[
p\in\Biggl[1-\dfrac{1}{\sqrt[n+1]{2}},1-\dfrac{1}{\sqrt[n]{2}}\Biggr).
\]
It follows that 
\[
\frac{1}{n+1}\leq -\log(1-p)<\frac{1}{n},
\]
and so by Corollary \ref{cor-dim}(1), we have $\dim(x)\geq\frac{1}{n+1}$.  (The case that $n=1$ can be shown by a nearly identical argument.)
\end{proof}

Theorem \ref{thm-int-dim} tells us how membership of some $x\in\cs$ in a $\mu_p^*$-random closed set from a family of random closed sets with a certain degree of intersectability puts a constraint on the effective dimension of $x$.  We now consider how the dimension of $x\in\cs$ puts a constraint on the degree of intersectability of any family of $\mu_p^*$-random closed sets such that  there is some $\mu_p^*$-random closed set $P$ with $x\in P$.  Hereafter, let us set $p_k=1-\frac{1}{\sqrt[k]{2}}$ for $k\geq 1$ (so that $p_1=1/2)$.

\begin{thm}\label{thm-dim-int} Suppose that $s\in(0,1]$ and $x\in\cs$ satisfies $\dim(x)=s$.

\medskip

\begin{enumerate}
\item For all $k\geq \lfloor\frac{1}{s}\rfloor$, $x$ is contained in a symmetric Bernoulli random closed set from a family with degree of intersectability $k$.

\bigskip

\item  If $s\neq \frac{1}{n}$ for $n\geq 1$, then $x$ is not contained in any symmetric Bernoulli random closed set from a family with degree of intersectability $k$ for $k<\lfloor\frac{1}{s}\rfloor$.
\end{enumerate}

\end{thm}

\begin{proof}
Given $x\in\cs$ with $\dim(x)=s$, let $n\in\omega$ satisfy $\frac{1}{n+1}<s\leq\frac{1}{n}$, so that  
$\lfloor\frac{1}{s}\rfloor=n$.  In particular,  for any $k\geq n$, we have $\dim(x)>\frac{1}{k+1}$.  By Corollary \ref{cor-dim}(2), it follows that $x$ is contained in a $\mu^*_{p_{k+1}}$-random closed set, which by Lemma \ref{lem-degint}(2) is in a family of random closed sets with degree of intersectability $k$, thereby establishing (1).

To show (2), suppose that $s\neq\frac{1}{n}$ for $n\geq 1$ and $x$ is in a symmetric Bernoulli random closed set from a family with degree of intersectability $k$ for some $k<\lfloor\frac{1}{s}\rfloor$, so that $k+1\leq\lfloor\frac{1}{s}\rfloor<\frac{1}{s}$.  Then by Lemma \ref{lem-degint}(2), $x$ is contained in a $\mu^*_p$-random closed set for some $p\in[p_{k+1},p_k)$, which by both parts of Corollary \ref{cor-dim} implies that 
\[
\dim(x)\geq\frac{1}{k+1}>s,
\]
which contradicts our assumption.
\end{proof}

We cannot improve Theorem \ref{thm-dim-int}(2) to cover the case that $s=\frac{1}{n}$ for $n\geq 2$ (the case for $n=1$ clearly cannot hold, since the the minimum value of the degree of intersectability of a family of random closed sets is 1).  For as shown by Diamondstone and Kjos-Hanssen in \cite{DiaKjo12}, there is some $x\in\cs$ with $\dim(x)=\frac{1}{2}$ and a $\mu_{p_2}^*$-random closed set $P$ such that $x\in P$.  Since the family of  $\mu_{p_2}^*$-random closed sets has  degree of intersectability equal to 1, this yields a counterexample to the possible extension of Theorem \ref{thm-dim-int}(2) for all $s\in(0,1]$.  In fact, with a slight modification of their proof, the Diamondstone/Kjos-Hanssen result can be readily generalized to hold for $s=\frac{1}{n}$ for all $n\geq 2$.

%
%

Note further that the generalization of the result due to Diamondstone/Kjos-Hanssen cited in the previous paragraph does not hold for all sequences $x$ satisfying $\dim(x)=\frac{1}{n}$ for some $n\geq 2$.  That is, for $n\geq 2$, it is not true that for $x\in\cs$,  $\dim(x)=\frac{1}{n}$ implies that $x$ is a member of a symmetric Bernoulli random closed set from a family with degree of intersectability equal to $n-1$.  By an unpublished result due to Jason Rute, the implications in Theorem \ref{thm-dkh} do not reverse.  Consequently, there exists, for instance, a sequence of dimension 1/2 that is not a member of any $\mu_{p_2}^*$-random closed set (the family of which has degree of intersectability equal to 1).  Thus, at best, we can conclude the following (by a proof nearly identical to the proof of Theorem \ref{thm-dim-int}(2)):

\begin{thm}
For $x\in\cs$ with $\dim(x)=\frac{1}{n}$ for some $n\geq 2$, $x$ is not contained in any symmetric Bernoulli random closed set from a family with degree of intersectability $k$ for $k<\lfloor\frac{1}{s}\rfloor-1$.
\end{thm}

\bibliographystyle{alpha} \bibliography{rcs}

\end{document}